\documentclass{amsart}
\newif\ifdraft
\draftfalse

\usepackage{amssymb,amsmath,amsthm,indentfirst,xspace,bm,bbold}
\usepackage{pdfsync}
\usepackage{framed}
\usepackage{mathtools}
\ifdraft
  \usepackage[notref,notcite]{showkeys}
\fi
\usepackage[normalem]{ulem}
\usepackage[pdfborder={0 0 .1},breaklinks,colorlinks=true]
           {hyperref}
\hypersetup{urlcolor=blue, citecolor=red}
\usepackage[initials,nobysame]{amsrefs}

\makeatletter
\def\@wraptoccontribs#1#2{}

\@mparswitchfalse
\makeatother

\ifdraft
    \newcommand{\sidenote}[1]{\pdfsyncstop\marginpar[\raggedleft\tiny #1]{\raggedright\tiny #1}\pdfsyncstart}
    
\else
    \newcommand{\sidenote}[1]{}
    
\fi
\newcommand{\warning}[1]{\typeout{}\typeout{WARNING: #1 at line \the\inputlineno}\typeout{}}
\newenvironment{todo}[1][TODO]{%
    \ifdraft\else\warning{TODO still present in final version}\fi
    \MakeFramed{\advance\hsize-\width \FrameRestore}\textbf{#1. }}%
    {\endMakeFramed}

  {\end{todo}
}

\makeatletter
\newcommand{\UWave}[2][blue]{\bgroup \markoverwith{\textcolor{#1}{\lower3.5\p@\hbox{\sixly \char58}}}\ULon{#2}}
\makeatother
\newcommand{\SOut}[2][red]{\bgroup\markoverwith {\textcolor{#1}{\rule[.45ex]{2pt}{.1ex}}}\ULon{#2}}

\newcommand{\highlight}[2][yellow]{\bgroup\markoverwith {\textcolor{#1}{\rule[-.2em]{2pt}{1.2em}}}\ULon{#2}}
\newcommand{\numberthis}{\addtocounter{equation}{1}\tag{\theequation}}

    {\left\{\begin{IEEEeqnarraybox}[\relax][c]{#1}}%
    {\end{IEEEeqnarraybox}\right.}

\newcommand{\grad}{\nabla}

\newcommand{\eps}{\varepsilon}
\renewcommand{\epsilon}{\eps}

\newcommand{\R}{\mathbb{R}}
\newcommand{\Z}{\mathbb{Z}}

\newcommand{\C}{\mathbb{C}}
\newcommand{\T}{\mathbb{T}}

\newif\iftextstyle
\textstyletrue
\everydisplay\expandafter{\the\everydisplay\textstylefalse}

\DeclarePairedDelimiter{\abs}{\lvert}{\rvert}
\DeclarePairedDelimiter{\norm}{\lVert}{\rVert}
\DeclarePairedDelimiter{\average}{\langle}{\rangle}

\newcommand{\ip}[2]{\average{#1,#2}}
\newcommand{\cip}[2]{\average{#1,#2}_\C}

\allowdisplaybreaks

\newtheoremstyle{compactplain}
  {\smallskipamount}
  {\smallskipamount}
  {\itshape}
  {}
  {\bfseries}
  {.}
  { }
  {}
\theoremstyle{plain}
\newtheorem{theorem}{Theorem}[section]
\newtheorem{lemma}[theorem]{Lemma}

\newtheorem*{theorem*}{Theorem}
\newtheorem*{lemma*}{Lemma}
\newtheorem*{proposition*}{Proposition}
\newtheorem*{corollary*}{Corollary}
\newtheorem*{problem*}{Problem}

\newtheoremstyle{noparen}
  {}
  {}
  {\itshape}
  {}
  {\bfseries}
  {.}
  { }
  {\thmname{#1}\thmnumber{ #2}\thmnote{ #3}}
\theoremstyle{noparen}

\theoremstyle{definition}

\theoremstyle{remark}
\newtheorem{remark}[theorem]{Remark}
\newtheorem*{remark*}{Remark}

\begin{document}
\title[Gaining two derivatives in the NSEs.]
      {Gaining two derivatives on a singular force in the 2D Navier-Stokes Equations.}
\author{Alexey Cheskidov}
    \address{Department of Mathematics, Statistics, and Computer Science\\
        University of Illinois at Chicago\\
        322 Science and Engineering Offices (M/C 249)\\
        851 S. Morgan Street\\
        Chicago, Illinois 60607-7045 USA}
    \email{acheskid@uic.edu}
\author{Landon Kavlie}
    \address{Department of Mathematics, Statistics, and Computer Science\\
        University of Illinois at Chicago\\
        322 Science and Engineering Offices (M/C 249)\\
        851 S. Morgan Street\\
        Chicago, Illinois 60607-7045 USA}
    \email{lkavli2@uic.edu}
\thanks{The authors were partially supported by NSF Grant DMS-1108864.}

\begin{abstract}

It has long been known, for the autonomous 2D Navier-Stokes equations with singular forcing in $H^{-1}$, that there exist unique solutions which are globally in $L^2$, a gain of one derivative. These classical techniques also show us that the solution is almost everywhere in $H^1$. On the other hand, if the forcing term is in $L^2$, it is known that the solution remains in $H^2$ globally, a gain of two derivatives. In this paper, we explore classical techniques to show that if the force is in $H^\alpha$ for $\alpha \in (-1,0)$, then the solution gains two derivatives globally. These methods break down for forces in $H^{-1}$. In this scenario, we use a Littlewood-Paley decomposition in Fourier space to show that a solution which exists in $H^1$ at some time $t$ must then remain in $H^1$ for a small interval of time $[t,t+\eps)$.
\end{abstract}

\maketitle

\section{Introduction}

The autonomous 2D Navier-Stokes equations on the torus $\T^2$ are given by 
\begin{equation}
  \label{nse}
  \left\{
    \begin{array}{l}
      u_t + (u \cdot \grad) u - \nu\Delta u + \grad p = f \\
      \grad \cdot u = 0.
    \end{array}
  \right.
\end{equation}
Due to the pioneering work of Leray (\cite{L34}), we know that the 2D Navier-Stokes equations admit unique solutions $u(t)$ with $u \in C([0,T],L^2) \cap L^2([0,T],H^1)$ for each $T \ge 0$ with a forcing term $f \in H^{-1}$. Moreover, in both two and three dimensions, these solutions satisfy the energy inequality. That is, for each $t \ge 0$, our solution $u(t)$ satisfies 
\begin{equation}
\label{energy_ineq}
  \|u(t)\|_2^2 + \nu \int_0^t \|u(s)\|_{H^1}^2 ds 
    \le \|u_0\|_2^2 + \frac{1}{\nu} t \|f\|_{H^{-1}}^2.
\end{equation}

With a ``smoother" force $f\in L^2$, then (\ref{nse}) admits unique solutions $u\in C([0,T],H^1) \cap L^\infty([0,T],H^2)$ for all times $T \ge 0$. For more information on these results, see \cite{CF88}, \cite{R01}, \cite{T84}, among others. On the other hand, for many linear parabolic equations, like the heat equation, we find that $u \in C([0,T],H^{\alpha+2})$ as long as $u(0) \in H^{\alpha+2}$ and $f \in H^\alpha$. The heat equation always gains two derivatives on the force, even with a singular force, such as $f \in H^{-1}$. 

This gap between the Navier-Stokes equations, and linear parabolic equations such as the heat equation was studied by Constantin and Seregin for the Navier-Stokes equations \cite{CS13a} and later for the Fokker-Plank equations \cite{CS13b}. Their analysis involves using the modulus of continuity in physical space. They find that with a forcing term $f\in W^{-1,q}$ with $q>4$, the solution $u$ remains H\"older continuous with exponent $1-4/q$. That is, in an $L^\infty$ sense, the function gains $2-\epsilon$ derivatives, for any $\epsilon \ge 0$ (with ``smoother forces" required to reach a two derivative gain). To further bridge this gap, our analysis uses the technique of Littlewood-Paley decompositions in Fourier space to show that $u$ remains in $H^1$ locally in time, a gain of two derivatives on the force $f\in H^{-1}=W^{-1,2}$. Specifically, we prove the following theorem.
\begin{theorem}
\label{main}
Let $u$ be the unique solution to (\ref{nse}) with $u(0):=u_0 \in H^1$. Then, there exists $T:=T(u_0,f)$ so that $u \in L^\infty([0,t_0], H^1)$ for all $0 < t_0 < T$.
\end{theorem}

Intervals where $u\in H^1$ are known as intervals of regularity for the Navier-Stokes equations. Using Theorem~\ref{main}, one can prove a Leray characterization for the 2D Navier-Stokes equations with force $f\in H^{-1}$. That is, $[0,\infty)=\cup_j [a_j,b_j)$ with $u(t)\in H^1$ for all $t\in[a_j,b_j)$. This result is well-known for the 3D Navier-Stokes equations with a force $f\in L^2$. A detailed discussion of these results in the three dimensional case can be found in \cite{CF88} or \cite{T84}, among others.

To complete our study, we explore the use of classical techniques in the intermediate spaces where $f \in H^\alpha$ with $\alpha \in (-1,0)$. In this setting, we show that classical techniques work to give a global gain of two derivatives. That is, we prove the following theorem: 
\begin{theorem}
\label{intermediate}
Let $f \in H^{\alpha}$ and $u(0) := u_0 \in H^{\alpha+2}$ for some $\alpha \in (-1,0)$. Then, there exists a solution $u(t)$ to (\ref{nse}) so that $u \in L^\infty([0,\infty), H^{\alpha+2})$.
\end{theorem}

The outline of our paper is as follows: In Section~\ref{preliminaries}, we introduce the classical setting for studying the Navier-Stokes equations, including the use of the Leray projection. We also recap the Littlewood-Paley decomposition of a function. In Section~\ref{heat}, we use the example of the heat equation to demonstrate our Fourier techniques for studying the Navier-Stokes equations in a ``simpler" setting. In Section~\ref{intermediate_case}, we use classical techniques to prove Theorem~\ref{intermediate}. To estimate the nonlinear term in the intermediate space $H^{\alpha}$ for $\alpha \in (-1,0)$, we use a Littlewood-Paley decomposition and Bony's paraproduct (\cite{B81}) as demonstrated in Lemma~\ref{nonlinearity_estimates}. Finally, we apply our Fourier analysis techniques to prove Theorem~\ref{main} in Section~\ref{proof}. 

Throughout this paper, we use the usual convention in which $C$ denotes an arbitrary, positive constant which may change from line to line. On the other hand, $C_0$, $C_1$, $C_2$, etc. are fixed constants.

\section{Preliminaries}
\label{preliminaries}

\subsection{The Projected Navier-Stokes Equations}

As this paper expands on classical theory, we will use classical notation as seen in (\cite{CF88},\cite{T84}) when exploring classical techniques. In Section~\ref{proof}, when employing harmonic analysis techniques, we use the unprojected Navier-Stokes equations (\ref{nse}) so that the reader may more easily track the elements such as derivatives. To begin, we apply the Leray projection $P_\sigma$ onto divergence-free vector fields to (\ref{nse}). For notational simplicity, we will assume that $P_\sigma f = f$. This gives us the projected equation 
\begin{equation}
\label{proj_nse}
  u_t + B(u,u) + \nu Au = f.
\end{equation}
Here, 
\begin{align*}
  B(u,u) &:= P_\sigma (u\cdot \grad) u \\
  Au      &:= P_\sigma (-\Delta) u.
\end{align*}

The Stokes operator $A$ is an unbounded, self-adjoint, positive definite operator with discrete eigenvalues $0 < \mu_1 \le \mu_2 \le \cdots$ (we avoid the classical use of $\lambda_p$ as that notation has a special designation which will be introduced in the next section). With corresponding eigenfunctions $w_j$, it is known that for $u = \sum_j a_j w_j$,
\[
  A u = \sum_j \mu_j a_j w_j.
\]
Thus, we may define the fractional Stokes operator as 
\[
  A^\alpha u := \sum_j \mu_j^\alpha a_j w_j
\]

Recall that the Sobolev space $H^\alpha(\T^2)$ for $\alpha \in \R$ is a Hilbert space with the norm 
\[
  \|u\|_{H^\alpha(\T^2)} := 
    \left(
      \sum_{n \in \Z^2} |n|^{2\alpha} |\hat{u}(n)|^2
    \right)^{1/2}
\]
where $\hat{u}(n)$ is the $n$th Fourier coefficient of $u$. Note that, using the operator $A$, an equivalent defintion of the $H^\alpha(\T^2)$ norm is given by
\[
  \|u\|_{H^\alpha(\T^2)} = \| A^{\alpha/2} u\|_{L^2(\T^2)}.
\]
For the remainder of this paper, we will omit $\T^2$ from our norms as it is implied. Moreover, we will use the standard convention
\[
  \|u\|_p := \|u\|_{L^p}
\]

\begin{remark}
It is customary to project the spaces $H^\alpha$ as well when using the projected equations. We will not make this distinction to more easily track the underlying space in which we are working. The fact that the evolution occurs in a projected space is coincidentally implied.
\end{remark}

\subsection{Littlewood-Paley Decomposition}

In this section, we briefly describe the Littlewood-Paley decomposition and the Littlewood-Paley theorem. This describes how to relate Sobolev norms in physical space via a particular breakdown in Fourier space. For more information on this theory, see, for example, the book by Chemin \cite{C98}, among others. 

Choose a nonnegative radial function $\chi \in C_0^\infty(\R^2)$ so that 
\[
  \chi(\xi) = 
  \left\{
    \begin{array}{ll}
      1, & \abs{\xi}\le\frac{1}{2} \\
      0, & \abs{\xi}>1.
    \end{array}
  \right.
\]
Let $\phi(\xi):=\chi(\lambda_1^{-1}\xi)-\chi(\xi)$. For each $q \ge 0$, we let 
$\phi_q(\xi):=\phi(\lambda_q^{-1}\xi)$. For technical reasons, let $\phi_{-1}(\xi):=\chi(\xi)$.

Given a tempered distribution vector field $u$ on $\T^2$ and $q \ge 1$, an integer, the $q$th Littlewood-Paley projection of $u$ is given by 
\[
  u_q(x) := \Delta_q u(x) := \sum_{k\in\Z^2}\hat{u}(k)\phi_q(k)e^{ik \cdot x}
\]
where $\hat{u}(k)$ is the $k$th Fourier coefficient of $u$. Note that, by the Littlewood-Paley theorem, 
\[
  \norm{u}_{H^s} \sim \left(\sum_{q=-1}^\infty\lambda_q^{2s}\norm{u_q}_2^2\right)^{1/2}
\]
for each $u \in H^s$ and $s\in\R$.

\section{The Heat Equation}
\label{heat}

To aid the reader in following our techniques for the Navier-Stokes equations in Section~\ref{proof}, we consider the $n$-dimensional heat equation on the torus $\T^n$. 
\begin{equation}
\label{heat_eqn}
  \left\{
    \begin{array}{l}
      u_t - \nu\Delta u = f \\
      u(0) = u_0.
    \end{array}
  \right.
\end{equation}
Suppose $u_0 \in H^{\alpha+2}$ and $f \in H^\alpha$ for some fixed $\alpha\in\R$. Then, we will prove the following Theorem.

\begin{theorem}
Let $u$ be a solution to (\ref{heat_eqn}). Then, $u \in L^\infty([0,\infty),H^{\alpha+2})$.
\end{theorem}

Note that, using the linearity of the heat equation, one can further show that $u \in C([0,T], H^{\alpha+2})$ for all $T \ge 0$.

The idea of our proof is to decompose our function using Littlewood-Paley theory. Then, apply Duhamel's formula to the decomposition before summing.

\begin{proof}
Multiply the first equation in (\ref{heat_eqn}) by $(u_q)_q:=\Delta_q(\Delta_q u)$ and integrate in space. This gives us that 
\[
  \frac{1}{2}\frac{d}{dt}\norm{u_q(t)}_2^2+\nu\lambda_q^2\norm{u_q(t)}_2^2
    \le \int_{\T^n} f_q \cdot u_q dx.
\]
Applying Cauchy-Schwartz followed by Young's inequality to the right-hand side, we get that 
\[
  \frac{d}{dt}\norm{u_q(t)}_2^2+\nu\lambda_q^2\norm{u_q(t)}_2^2
    \le \frac{1}{\nu\lambda_q^2}\norm{f_q}_2^2
\]

Next, we apply Duhamel's principle to get that 
\begin{align*}
  \norm{u_q(t)}_2^2 
    &\le \norm{u_q(0)}_2^2 e^{-\nu\lambda_q^2 t}
           +\frac{1}{\nu\lambda_q^2}\norm{f_q}_2^2
           \int_0^t e^{\nu\lambda_q^2(s-t)} ds. \\
    &\le \norm{u_q(0)}_2^2 e^{-\nu\lambda_q^2 t} 
           +\frac{1}{\nu\lambda_q^4}\norm{f_q}_2^2
           \left[1-e^{-\nu\lambda_q^2 t} \right].
\end{align*}
Multiplying by $\lambda_q^{2\alpha+4}$ and summing in $q$ yields the result.
\end{proof}

\section{The \texorpdfstring{$H^\alpha$}{H\^{}a} Case for \texorpdfstring{$\alpha \in (-1,0)$}{a in (-1,0)}}
\label{intermediate_case}

We begin our study of the 2D Navier-Stokes equations by investigating classical techniques. Here, we show that classical arguments yeild a uniform two derivative gain for $f \in H^\alpha$ for all $\alpha \in (-1,0)$. Note that when $\alpha := 0$ ($f \in L^2$), it is known classically that $u(t) \in H^2$ for all $t \ge 0$, a two derivative gain. We begin by estimating the nonlinear term in $H^\alpha$.

\subsection{Estimating the Nonlinear Term}

\begin{lemma}
\label{nonlinearity_estimates}
Let $u \in H^1 \cap H^{1-\beta}$ for $\beta \in (0,1)$. Then, 
\[
  \|B(u,u)\|_{H^{-\beta}} \le C \|u\|_{H^{1-\beta}} \|u\|_{H^1}.
\]
\end{lemma}

Note that when $\beta = 1$, the estimate becomes 
\[
  \|B(u,u)\|_{H^{-1}} \le C \|u\|_2 \|u\|_{H^1}.
\]
This is the classical estimate which is obtained using H\''older's inequality and the Ladyzhenskaya inequality. On the other hand, when $\beta = 0$, we may use interpolation to say that 
\[
  \|B(u,u)\|_2 \le C \|u\|_{H^1} \|u\|_{H^1} 
                   \le C \|u\|_2^{1/2} \|u\|_{H^2}^{1/2} \|u\|_{H^1}.
\]
This is the classical estimate which is obtained using H\''older's inequality followed by Agmon's inequality. Thus, our estimate accurately generalizes the classical estimates.

\begin{proof}
Let $v \in H^\beta$. We must estimate the integral
\[
  \int_{\T^2} u \cdot \grad u \cdot v dx.
\]
To do so, we will use Bony's paraproduct. Separating each term into it's Littlewood-Paley pieces, we apply the necessary cancellations to find that 
\begin{align*}
  \left| \int_{\T^2} u \cdot \grad u \cdot v dx \right|
    \le&    \sum_{\substack{|p-q|\le 2 \\ r < q+1}} 
              \int_{\T^2} \left| u_p \cdot \grad u_q \cdot v_r \right| dx \\
       & +  \sum_{\substack{|p-r|\le 2 \\ q < r+1}} 
              \int_{\T^2} \left| u_p \cdot \grad u_q \cdot v_r \right| dx \\
       & +  \sum_{\substack{|q-r|\le 2 \\ p < r+1}} 
              \int_{\T^2} \left| u_p \cdot \grad u_q \cdot v_r \right| dx \\
       & =: I + II + III.
\end{align*}

To estimate $I$, we use H\"older's inequality followed by Bernstein's inequality to find that 
\[
  I \le \sum_{\substack{|p-q|\le 2 \\ r < p+1}} 
           \| u_p \|_2 \| \grad u_q \|_2 \| v_r \|_\infty
    \le C \| u \|_{H^1} \sum_{r < p+1} \| u_p \|_2 \lambda_r \| v_r \|_2.
\]
Splitting the derivative $\lambda_r$ and applying Cauchy-Schwarz inequality gives us that
\[
  I \le C \| u \|_{H^1} 
               \underbrace{\left(\sum_{r < p+1} \lambda_r^{2-2\beta-2\epsilon} 
                 \lambda_p^{2\epsilon} \| u_p \|_2^2\right)^{1/2}}_{=: I_A}
               \underbrace{\left(\sum_{r < p+1} \lambda_r^{2\beta+2\epsilon} 
                 \lambda_p^{-2\epsilon} \| v_r \|_2^2 \right)^{1/2}}_{=: I_B}
\]
where $\epsilon \ll 1$ is chosen so that $\beta+\epsilon < 1$. 

For the first sum, $I_A^2$, we see that 
\begin{align*}
  I_A^2 &=   \sum_{p=-1}^\infty\sum_{r=-1}^p \lambda_r^{2-2\beta-2\epsilon}
               \lambda_p^{2\epsilon} \| u_p \|_2^2  \\
        &\le C \sum_{p=-1}^\infty \lambda_p^{2-2\beta} \| u_p \|_2^2  \\
        &\le C \| u \|_{H^{1-\beta}}^2.
\end{align*}
For the second sum, $I_B^2$, we must switch the order of summation as show below:
\begin{align*}
  I_B^2 &=   \sum_{p=-1}^\infty\sum_{r=-1}^p \lambda_r^{2\beta+2\epsilon}
               \lambda_p^{-2\epsilon} \| v_r \|_2^2 \\
        &=   \sum_{r=-1}^\infty \sum_{p=r}^\infty \lambda_r^{2\beta+2\epsilon}
               \lambda_p^{-2\epsilon} \| v_r \|_2^2  \\
        &\le C \sum_{r=-1}^\infty \lambda_r^{2\beta} \| v_r \|_2^2  \\
        &\le C \| v \|_{H^{\beta}}^2.
\end{align*}

To estimate $II$, we proceed as with $I$ using H\"older's inequality followed by Bernstein's inequality to give us that 
\[
  II \le \sum_{\substack{|p-r|\le 2 \\ q < r+1}} 
            \| u_p \|_2 \| \grad u_q \|_\infty \| v_r \|_2 
     \le C \| u \|_{H^1} \sum_{q < r+1} \lambda_q \| u_q \|_2 \| v_r \|_2.
\]
Similarly with $I$, we split the derivative $\lambda_q$ and apply Cauchy-Schwarz to get
\[
  II \le C \| u \|_{H^1} 
               \underbrace{\left(\sum_{q < r+1} \lambda_q^{2-2\beta+2\epsilon} 
                 \lambda_r^{-2\epsilon} \| u_q \|_2^2\right)^{1/2}}_{=: II_A}
               \underbrace{\left(\sum_{q < r+1} \lambda_q^{2\beta-2\epsilon} 
                 \lambda_p^{2\epsilon} \| v_r \|_2^2 \right)^{1/2}}_{=: II_B}
\]
Switching the order of summation in $II_A$ and proceeding as in $I$, we find that 
\begin{align*}
  II_A^2 &\le C \| u \|_{H^{1-\beta}}^2 \\
  II_B^2 &\le C \| v \|_{H^{\beta}}^2.
\end{align*}

Finally, to estimate $III$, we again use H\"older's inequality followed by Bernstein's inequality to get
\[
  II \le \sum_{\substack{|q-r| \le 2 \\ p < r+1}}
           \| u_p \|_\infty \| \grad u_q \|_2 \| v_r \|_2
     \le C \| u \|_{H^1} \sum_{p < r+1} \lambda_p \| u_p \|_2 \| v_r \|_2.
\]
The rest of the estimates for $III$ proceed exactly as in the case for $II$. 
\end{proof}

\subsection{Gaining One Derivative}

In this section, as well as the following section, we make {\it a priori} estimates. The calculations are done on the level of Galerkin approximations. One can then pass to the limit to obtain the stated bounds for the actual solutions.

\begin{theorem}
\label{one_deriv_gain}
Let $f \in H^\alpha$ and  $u(0) := u_0 \in H^{\alpha+1}$ for some $\alpha \in (-1,0)$. Then, there exists a solution $u(t)$ to (\ref{proj_nse}) so that $u \in L^\infty([0,\infty), H^{\alpha+1})$.
\end{theorem}

\begin{proof}

Taking the inner product of (\ref{proj_nse}) with $A^{\alpha+1} u$ and integrating in space, we find that 
\begin{align*}
  \frac{1}{2} \frac{d}{dt} \|u\|_{H^{\alpha+1}}^2 + \nu \|u\|_{H^{\alpha+2}}^2
    &\le \int_{\T^2} \left| B(u,u) \cdot A^{\alpha+1} u \right| dx
       + \int_{\T^2} \left| f \cdot A^{\alpha+1} u \right| dx \\
    &\le \|B(u,u)\|_{H^\alpha} \|u\|_{H^{\alpha+2}}
       + \|f\|_{H^\alpha} \|u\|_{H^{\alpha+2}}.
\end{align*}

Using Lemma~\ref{nonlinearity_estimates} and Young's inequality, we find that
\begin{align}
\label{Ha_est1}
  \frac{1}{2} \frac{d}{dt} \|u\|_{H^{\alpha+1}}^2 
      + \nu \|u\|_{H^{\alpha+2}}^2
    &\le C \|u\|_{H^1} \|u\|_{H^{\alpha+1}} \|u\|_{H^{\alpha+2}}
         + \|f\|_\alpha \|u\|_{H^{\alpha+2}} \\
  \frac{d}{dt} \|u\|_{H^{\alpha+1}}^2 
      + \nu \|u\|_{H^{\alpha+2}}^2
    &\le \frac{C}{\nu} \|u\|_{H^1}^2 \|u\|_{H^{\alpha+1}}^{2}
         + \frac{2}{\nu} \|f\|_\alpha^2.
\end{align}
Dropping the $H^{\alpha+2}$ term, and using Gronwall, we have that for $t \ge t_0 \ge 0$, 
\[
  \|u(t)\|_{H^{\alpha+1}}^2 \le \left(\|u(t_0)\|_{H^{\alpha+1}}^2 
                                  +  \frac{2}{\nu} \|f\|_{H^\alpha}^2(t-t_0) \right)
            \exp\left(\frac{C}{\nu}\int_{t_0}^t \|u(s)\|_{H^1}^2 ds \right).
\]
After using the embedding $H^1 \subset H^{\alpha+1}$, this becomes 
\begin{equation*}
  \|u(t)\|_{H^{\alpha+1}}^2 \le \left(\|u(t_0)\|_{H^1}^2 
                                  +  \frac{2}{\nu} \|f\|_{H^\alpha}^2(t-t_0) \right)
            \exp\left(\frac{C}{\nu}\int_{t_0}^t \|u(s)\|_{H^1}^2 ds \right).
\end{equation*}

By the energy inequality (\ref{energy_ineq}), we know that 
\[
  \nu \int_{t_0}^t \|u(s)\|_{H^1}^2 ds \le \|u_0\|_2^2
                                        +   \frac{1}{\nu}\|f\|_{H^{-1}}(t-t_0).
\]
Therefore, for $0 < \epsilon \le t$, we have that 
\[
  \left| \left\{ t_0 \in [t-\epsilon, t] : \|u(t_0)\|_{H^1} \ge M \right\} \right|
      \le \frac{1}{M^2} \left( \frac{1}{\nu} \|u_0\|_2^2 
                             + \frac{\epsilon}{\nu^2} \|f\|_{H^{-1}}^2\right)
\]

Letting 
\[
M := \sqrt{\frac{2}{\epsilon}\left( \frac{1}{\nu} \|u_0\|_2^2 
               + \frac{\epsilon}{\nu^2} \|f\|_{H^{-1}}^2\right)},
\]
we find that 
\[
  \left| \left\{ t_0 \in [t-\epsilon, t] : \|u(t_0)\|_{H^1} \ge M \right\} \right|
      \le \frac{\epsilon}{2}.
\]
Therefore, there exists $t_0 \in [t-\epsilon,t]$ so that 
\[
  \|u(t_0)\|_{H^1}^2 \le \frac{2}{\epsilon} \left( \frac{1}{\nu} \|u_0\|_2^2
                  + \frac{\epsilon}{\nu^2} \|f\|_{H^{-1}}^2 \right).
\]
So, the above argument shows us that for $t \ge \epsilon$,
\begin{align*}
  \|u(t)\|_{H^{\alpha+1}}^2 
    &\le \left(\frac{2}{\epsilon\nu} \|u_0\|_2^2
                             + \frac{2}{\nu^2} \|f\|_{H^{-1}}^2 
                             +  \frac{2\epsilon}{\nu} \|f\|_{H^\alpha}^2 \right)
                             e^{C_0} \\
    &\le \left(\frac{2}{\epsilon\nu} \|u_0\|_{H^{\alpha+1}}^2
                             + \frac{2}{\nu^2} \|f\|_{H^\alpha}^2 
                             +  \frac{2\epsilon}{\nu} \|f\|_{H^\alpha}^2 \right)
                             e^{C_0} 
\end{align*}
for 
\[
 C_0 := C \|u_0\|_{H^{\alpha+1}}^2 + \frac{C\epsilon}{\nu}\|f\|_{H^\alpha}^2.
\]

This gives boundedness of $u$ in $H^{\alpha+1}$ for all $t \ge \epsilon$ for any fixed $\epsilon > 0$. To show the boundedness of $u$ in $H^{\alpha+1}$ for small $t$, we go back to equation (\ref{Ha_est1}). Using interpolation, we have that 
\[
  \|u\|_{H^1} \le C \|u\|_2^{\frac{1}{\alpha+2}} 
                    \|u\|_{H^{\alpha+2}}^{\frac{\alpha+1}{\alpha+2}}.
\]
Thus, after using interpolation and Young's inequality on (\ref{Ha_est1}), we have that
\[
  \frac{1}{2} \frac{d}{dt} \|u\|_{H^{\alpha+1}}^2 
      + \nu \|u\|_{H^{\alpha+2}}^2
      \le \frac{C}{\nu^{2\alpha+3}} \|u\|_2^2 \|u\|_{H^{\alpha+1}}^{2\alpha+4}
         + \frac{1}{\nu}\|f\|_\alpha^2 \\
\]
Dropping the $H^{\alpha+2}$ term, we may use nonlinear Gronwall to say that the $H^{\alpha+1}$ norm remains bounded for small time. 

Combining this short-term bound with the previous long-term bound gives us that $u \in L^\infty([0,\infty), H^{\alpha+1})$, as required.

\end{proof}

\subsection{Gaining Two Derivatives}

We combine the result of the previous section with an analyticity argument to show the uniform gain of two derivatives in this regime. That is, we prove Theorem~\ref{intermediate} which we restate here for the reader's convenience.

\begin{theorem}
Let $f \in H^{\alpha}$ and $u(0) := u_0 \in H^{\alpha+2}$ for some $\alpha \in (-1,0)$. Then, there exists a solution $u(t)$ to (\ref{proj_nse}) so that $u \in L^\infty([0,\infty), H^{\alpha+2})$.
\end{theorem}

To use analyticity arguments, we first need to complexify the spaces $H^\alpha$ as well as the Navier-Stokes equations themselves. First, the complexified space $H^\alpha_\C$ is given by 
\[
  H^\alpha_\C := \{u = u_1 + i u_2 : u_1, u_2 \in H^\alpha\}
\]
with the inner product defined via linearity as 
\[
  \ip{u_1 + i u_2}{v_1 + i v_2}_{H^\alpha_\C} 
    := \ip{u_1}{v_1}_{H^\alpha} + \ip{u_2}{v_2}_{H^\alpha} 
     + i (\ip{u_2}{v_1}_{H^\alpha} - \ip{u_1}{v_2}_{H^\alpha}). 
\]
When we write $\cip{u}{v}$, we are using the usual functional pairing, in a complex sense.

We will let the time $t := se^{i\theta}$. It is known, in this setting, that there exist unique, analytic solutions to the Galerkin system for complex time $t$ in some neighborhood of the origin. Moreover, the restriction of these solutions to the real line agree with the usual Galerkin approximations.

\begin{proof}
To begin, we multiply the complexified Navier-Stokes equations by $e^{i\theta}$, take the inner product with $A^{\alpha+1}u$, and take the real part. This gives us
\begin{align*}
  \frac{1}{2} \frac{d}{ds} \|u(se^{i\theta})\|_{H^{\alpha+1}}^2 
    &+ \nu \cos(\theta) \|u(se^{i\theta})\|_{H^{\alpha+2}}^2 \\
    &= \text{Real} 
         \left[
            e^{i\theta}(\cip{B(u,u)}{A^{\alpha+1}u}+\cip{f}{A^{\alpha+1}u})
         \right].
\end{align*}
Estimating the right-hand side, we first see, as in the real case, that 
\[
  \left|\cip{f}{A^{\alpha+1}u}\right|
    \le \|f\|_{H^\alpha} \|u\|_{H^{\alpha+2}}
    \le \frac{1}{\nu \cos(\theta)} \|f\|_{H^\alpha}^2 
      + \frac{\nu \cos(\theta)}{4} \|u\|_{H^{\alpha+2}}^2
\]
Next, we use Lemma~\ref{nonlinearity_estimates} to see that 
\begin{align*}
  \left|\cip{B(u,u)}{A^{\alpha+1}u}\right|
    &\le \|B(u,u)\|_{H^\alpha} \|u\|_{H^{\alpha+2}} \\
    &\le C \|u\|_{H^{\alpha+1}}^2 
           \|u\|_{H^{\alpha+2}}^{\frac{2\alpha+3}{\alpha+2}} \\
    &\le \frac{C}{(\nu \cos(\theta))^{2\alpha+3}} \|u\|_{H^{\alpha+1}}^{2\alpha+6} 
       + \frac{\nu \cos(\theta)}{4} \|u\|_{H^{\alpha+2}}^2.
\end{align*}
Thus, we obtain the Riccati-type inequality 
\[
  \frac{d}{dt} \|u\|_{H^{\alpha+1}}^2 + \nu \cos(\theta) \|u\|_{H^{\alpha+2}}^2
    \le \frac{2}{\nu \cos(\theta)} \|f\|_{H^\alpha}^2 
      + \frac{C}{(\nu \cos(\theta))^{2\alpha+3}} \|u\|_{H^{\alpha+1}}^{2\alpha+6}.
\]

This inequality shows us that for some time $\|u(t)\|_{H^{\alpha+1}} \le M$ for some fixed $M > 0$ and $t \le T := T(\|u_0\|_{H^{\alpha+1}},\nu,f,\theta)$. Therefore, the solutions to the complexified Navier-Stokes equations extend to analytic solutions in a neighborhood $D$ of the origin  given by 
\[
  D := \{t = se^{i\theta} : 0 < s < T, |\theta|<\pi/2\}.
\]
Note that $D$ is symmetric across the real axis, by construction. Also, note that within $D$, $\|u(t)\|_{H^{\alpha+1}} < M$.

Fix a compact set $K \subset D$. By Cauchy's formula with $\gamma$ a circle in $K$ of radius $r < d(K, \partial D)$, we have that 
\[
  \frac{du}{dt}(t) = \frac{1}{2\pi i} \int_{\gamma} \frac{u(z)}{(z-t)^2} dz
\]
for all $t \in K$. Taking the $H^{\alpha+1}$ norm of this equation, we get that 
\[
  \|u_t(t)\|_{H^{\alpha+1}} \le \frac{M}{r}.
\]

Within $K$, we find that 
\begin{align*}
  \nu \|A u\|_{H^\alpha} 
                \numberthis \label{two_deriv_bound}
    &\le \|u_t\|_{H^\alpha} + \|B(u,u)\|_{H^\alpha} + \|f\|_{H^\alpha} \\
    &\le \mu_1 \|u_t\|_{H^{\alpha+1}} + C\|u\|_{H^1}\|u\|_{H^{\alpha+1}} 
                                + \|f\|_{H^\alpha} \\
    &\le \mu_1 \|u_t\|_{H^{\alpha+1}} + C\|u\|_2^{\frac{1}{\alpha+2}} 
                                  \|u\|_{H^{\alpha+1}} 
                                  \|u\|_{H^{\alpha+2}}^{\frac{\alpha+1}{\alpha+2}}
                                + \|f\|_{H^\alpha} \\
    &\le \mu_1 \|u_t\|_{H^{\alpha+1}} + \frac{C}{\nu^{\alpha+1}}\|u\|_2 
                                  \|u\|_{H^{\alpha+1}}^{\alpha+2} 
                                + \frac{\nu}{2} \|u\|_{H^{\alpha+2}}
                                + \|f\|_{H^\alpha}
\end{align*}
where we used the Poincar\'e inequality with the Poincar\'e constant $\mu_1$ along with Lemma~\ref{nonlinearity_estimates} in the second line; we used interpolation in the second line; and we used Young's inequality in the final line. 

Moving the $H^{\alpha+2}$ terms to the same side of the equation and noting that $\|u_t\|_{H^{\alpha+1}}$ is bounded by analyticity, $\|u\|_2$ is bounded by the energy inequality (\ref{energy_ineq}), and $\|u\|_{H^{\alpha+1}}$ is bounded by Theorem~\ref{one_deriv_gain}, we now know that $\|u\|_{H^{\alpha+2}}$ is bounded in the compact set $K$.

Using the uniform boundedness of $\|u\|_{H^{\alpha+1}}$ obtained in Theorem~\ref{one_deriv_gain}, we can rerun this argument with the same bounds at each starting point $t_0 \in [0,\infty)$. Thus, $\|u\|_{H^{\alpha+2}}$ is uniformly bounded in a complex neighborhood of the real axis. In particular, $\|u(t)\|_{H^{\alpha+2}} < C < \infty$ for each $t \in [0,\infty)$.
\end{proof}

\section{The \texorpdfstring{$H^{-1}$}{H\^{}-1} Case}
\label{proof}

In this section, we present the proof of Theorem~\ref{main} which we again restate here for the reader's convenience.

\begin{theorem}
Let $u$ be the unique solution to (\ref{nse}) with $u(0):=u_0 \in H^1$. Then, there exists $T:=T(u_0,f)$ so that $u \in L^\infty([0,t_0], H^1)$ for all $0 < t_0 < T$.
\end{theorem}

The reason analyticity methods from the previous section fails are as follows: Note that when $\alpha = -1$, the inequality in (\ref{two_deriv_bound}) becomes
\[
  \nu \|A u\|_{H^{-1}}
    \le \|u_t\|_{H^{-1}} + C\|u\|_{H^1} \|u\|_{H^2} + \|f\|_{H^{-1}}
\]
since interpolating between the $H^1$ norm between $L^2$ and $H^{\alpha+2}$ fails. We are now unable to use Young's inequality to split the $H^1$ and $H^2$ norms to move the $\|u\|_{H^2}$ terms to the left-hand side. Thus, we must use another tactic.

Again, note that in this section, we use the unprojected Navier-Stokes equations so that the use of harmonic analysis techniques are more easily followed.

\begin{proof}
Multiply the first equation in (\ref{nse}) with $(u_q)_q:=\Delta_q\left(\Delta_q u\right)$ and integrate in space. This becomes 
\begin{align*}
  \frac{1}{2}\frac{d}{dt}\|u_q(t)\|_2^2+\nu\lambda_q^2\|u_q(t)\|_2^2 = 
    \int_{\T^2} (u\cdot\grad u)\cdot (u_q)_q dx+\int_{\T^2} f_q\cdot u_q dx.
\end{align*}
Apply Cauchy-Swartz and Young's inequality to the second term on the right-hand side. This gives
\begin{align*}
  \frac{1}{2}\frac{d}{dt}\|u_q(t)\|_2^2+\nu\lambda_q^2\|u_q(t)\|_2^2 
    \le& \int_{\T^2} u\cdot\grad u\cdot (u_q)_q dx
           + \frac{\nu\lambda_q^2}{4}\|u_q(t)\|_2^2 
           + \frac{1}{\nu\lambda_q^2}\|f_q\|_2^2.
\end{align*}

For the first term on the right-hand side, use H\"older's inequality. 
\begin{equation*}
  \int_{\T^2} u \cdot \grad u \cdot (u_q)_q 
      \le C\|u\|_r \|u\|_{H^1} \|u_q\|_\rho
  \qquad \frac{1}{r}+\frac{1}{\rho}=\frac{1}{2}.
\end{equation*}
Assume that $2<r<\infty$. Then, Applying the Sobolev and Bernstein inequalities give us that 
\begin{align*}
  \frac{1}{2}\frac{d}{dt}\|u_q(t)\|_2^2
      +\frac{3}{4}\nu\lambda_q^2\|u_q(t)\|_2^2 
    \le& C\|u(t)\|_r \|u(t)\|_{H^1} \|u_q(t)\|_\rho
            +\frac{1}{\nu\lambda_q^2} \|f_q\|_2 \\
    \le& C\|u(t)\|_{H^1}^2 \lambda_q^{(\rho-2)/2\rho} \|u_q(t)\|_2
            +\frac{1}{\nu\lambda_q^2} \|f_q\|_2^2 \\
\end{align*}

For simplicity of notation, let $p\in(0,1)$ be given by $p:=(\rho-2)/2\rho$. Then, using Young's inequality, we find that 
\begin{equation*}
  \frac{d}{dt}\|u_q(t)\|_2^2+\nu\lambda_q^2\|u_q(t)\|_2^2
    \le \frac{C}{\nu\lambda_q^{2-2p}} \|u(t)\|_{H^1}^4
            +\frac{2}{\nu\lambda_q^2} \|f_q\|_2^2.
\end{equation*}

Next, we apply Duhamel's principle. This gives 
\begin{align*}
  \|u_q(t)\|_2^2
    \le e^{-\nu\lambda_q^2t} \|u_q(0)\|_2^2
                \numberthis \label{integrated_force}
       &   +\frac{2}{\nu^2}\lambda_q^{-4} \|f_q\|_2^2
             \left[1-e^{-\nu\lambda_q^2t}\right] \\
       &   +\frac{C}{\nu}\int_0^te^{\nu\lambda_q^2(s-t)}\lambda_q^{2p-2}
             \|u(s)\|_{H^1}^4 ds
\end{align*}
Multiply this through by $\lambda_q^2$ and sum in $q$. We find that 
\begin{equation}
\label{H1_est}
  \|u(t)\|_{H^1}^2 \le \|u_0\|_{H^1}^2
    +\frac{2}{\nu^2} \|f\|_{H^{-1}}^2
    +\frac{C}{\nu}\int_0^t\sum_qe^{\nu\lambda_q^2(s-t)}\lambda_q^{2p}
        \|u(s)\|_{H^1}^4 ds.
\end{equation}

\begin{remark}
It is worthwhile to note that inequality (\ref{integrated_force}), obtained via integrating the force term, can also be obtained using a non-autonomous (or time-dependent) force. To obtain this, we need $f \in L^\infty_{loc}([0,\infty),H^{-1})$ with a ``dominating function in Fourier space." We mean that there exists a $g \in H^{-1}$ with 
\[
  \|f_q(t)\|_2 \le \|g_q\|_2
\]
for all $t \ge 0$ and $q \ge Q$ for some finite integer $Q \ge -1$.
\end{remark}

Taking a closer look at the final integral, consider the sum 
\begin{equation}
\label{sum}
  \sum_q e^{\nu\lambda_q^2(s-t)}\lambda_q^{2p}.
\end{equation}
Fix a constant $\gamma > 0$ to be determined later. Then, let $Q_0 > 1$ be chosen so that $\ln \lambda_q \le (\lambda_q)^\gamma$ for all $q \ge Q_0$. Then, define 
\begin{align*}
  Q(s) & :=\min\left\{q \ge Q_0:
           e^{\nu\lambda_q^2(s-t)} \le \lambda_q^{-2p-1}\right\} \\
  \Lambda(s) & := \lambda_{Q(s)}.
\end{align*}
We estimate the integral as follows
\[
  \int_0^t\sum_qe^{\nu\lambda_q^2(s-t)}\lambda_q^{2p}
          \|u(s)\|_{H^1}^4
    \le \underbrace{\int_{[0,t]\cap\mathbb{1}_{Q(s)\le Q_0}}}_{:= I}
        +\underbrace{\int_{[0,t]\cap\mathbb{1}_{Q(s)> Q_0}}}_{:=II}.
\]

For $I$, we have that 
\begin{align*}
  I &\le \int_0^t \|u(s)\|_{H^1}^4
         \left(\sum_{q\le Q_0} e^{\nu\lambda_q^2(s-t)}\lambda_q^{2p}
             + \sum_{q > Q_0}\lambda_q^{-2}\right)ds \\
    &\le \int_0^t\|u(s)\|_{H^1}^4 (Q_0\Lambda_{Q_0}^{2p}+1)ds \\
    &\le C\int_0^t\|u(s)\|_{H^1}^4ds.
\end{align*}

For $II$, we see that for a fixed $s$, (\ref{sum}) can be estimated by 
\begin{equation*}
  \sum_{q \le Q(s)}e^{\nu\lambda_q^2(s-t)}\lambda_q^{2p}
    +\sum_{q>Q(s)}\lambda_q^{-1}
  \le Q(s)\Lambda(s)^{2p}+1.
\end{equation*}
By the definition of $\Lambda$, $2^{-1}\Lambda$ satisfies
\[
  2^{2p+1}\Lambda^{-2p-1}\le e^{\nu2^{-2}\Lambda^2(s-t)}
    \iff 
  \Lambda^2 \le \frac{8(p + 1/2)}{\nu(t-s)}\ln\Lambda.
\]
But, $\ln\Lambda\le\Lambda^\gamma$ by definition. Thus, 
\[
  \Lambda^{2-\gamma} \le \frac{8(p + 1/2)}{\nu(t-s)}.
\]
for $\gamma<2$.

Proceding in much the same way as we did with $I$, we see that 
\begin{align*}
  II &\le   \int_0^t \|u(s)\|_{H^1}^4
            \left(\sum_{q\le Q(s)} \Lambda(s)^{2p}
                + \sum_{q > Q(s)}\lambda_q^{-2}\right) ds \\
     &\le C \int_0^t \|u(s)\|_{H^1}^4 (Q(s)\Lambda(s)^{2p}+1) ds \\
     &\le C \int_0^t \|u(s)\|_{H^1}^4 (\ln\Lambda(s)
                                          \Lambda(s)^{2p}+1) ds \\
     &\le C \int_0^t \|u(s)\|_{H^1}^4 (\Lambda(s)^{2p+\gamma}+1) ds \\
     &\le C \int_0^t \left(\frac{1}
              {(\nu(t-s))^{(2p+\gamma)/(2-\gamma)}}+1\right)
              \|u(s)\|_{H^1}^4 ds.
\end{align*}
Putting these estimates together with (\ref{H1_est}), we have that 
\begin{align*}
  \|u(t)\|_{H^\alpha}^2 \le \|u_0\|_{H^\alpha}^2
    +& \frac{2}{\nu^2} \|f\|_{H^{\alpha-2}}^2  \\
    +& \frac{C}{\nu} \int_0^t \left(\frac{1}
              {(\nu(t-s))^{(2p+\gamma)/(2-\gamma)}}+1\right)
              \|u(s)\|_{H^1}^4 ds.
\end{align*}

Next, let $p:=1/4$ and $\gamma:=1/2$. Then, $(2p+\gamma)/(2-\gamma)=2/3$. Note that this choice of $\gamma$ means that $Q_0 = 2$. This gives us that 
\begin{align*}
  \|u(t)\|_{H^1}^2 \le \|u_0\|_{H^1}^2
    +&\frac{2}{\nu^2} \|f\|_{H^{-1}}^2 \\
    +&\frac{C}{\nu} \int_0^t \left((\nu(t-s))^{-2/3}+1\right)
        \|u(s)\|_{H^1}^4 ds.
\end{align*}
An application of nonlinear Gronwall leads to the desired result.
\end{proof}

\bibliographystyle{plain}
\bibliography{refs}

\begin{bibdiv}
\begin{biblist}

\bib{B81}{article}{
      author={Bony, Jean-Michel},
       title={Calcul symbolique et propagation des singularit\'es pour les
  \'equations aux d\'eriv\'ees partielles non lin\'eaires},
        date={1981},
        ISSN={0012-9593},
     journal={Ann. Sci. \'Ecole Norm. Sup. (4)},
      volume={14},
      number={2},
       pages={209\ndash 246},
         url={http://www.numdam.org/item?id=ASENS_1981_4_14_2_209_0},
      review={\MR{631751 (84h:35177)}},
}

\bib{C98}{book}{
      author={Chemin, Jean-Yves},
       title={Perfect incompressible fluids},
      series={Oxford Lecture Series in Mathematics and its Applications},
   publisher={The Clarendon Press, Oxford University Press, New York},
        date={1998},
      volume={14},
        ISBN={0-19-850397-0},
        note={Translated from the 1995 French original by Isabelle Gallagher
  and Dragos Iftimie},
      review={\MR{1688875 (2000a:76030)}},
}

\bib{CF88}{book}{
      author={Constantin, Peter},
      author={Foias, Ciprian},
       title={Navier-{S}tokes equations},
      series={Chicago Lectures in Mathematics},
   publisher={University of Chicago Press},
     address={Chicago, IL},
        date={1988},
        ISBN={0-226-11548-8; 0-226-11549-6},
      review={\MR{972259 (90b:35190)}},
}

\bib{CS13b}{article}{
      author={Constantin, Peter},
      author={Seregin, Gregory},
       title={Global regularity of solutions of coupled {N}avier-{S}tokes
  equations and nonlinear {F}okker {P}lanck equations},
        date={2010},
        ISSN={1078-0947},
     journal={Discrete Contin. Dyn. Syst.},
      volume={26},
      number={4},
       pages={1185\ndash 1196},
         url={http://dx.doi.org/10.3934/dcds.2010.26.1185},
      review={\MR{2600741 (2011i:35188)}},
}

\bib{CS13a}{incollection}{
      author={Constantin, Peter},
      author={Seregin, Gregory},
       title={H\"older continuity of solutions of 2{D} {N}avier-{S}tokes
  equations with singular forcing},
        date={2010},
   booktitle={Nonlinear partial differential equations and related topics},
      series={Amer. Math. Soc. Transl. Ser. 2},
      volume={229},
   publisher={Amer. Math. Soc., Providence, RI},
       pages={87\ndash 95},
      review={\MR{2667634 (2012a:35225)}},
}

\bib{L34}{article}{
      author={Leray, Jean},
       title={Sur le mouvement d'un liquide visqueux emplissant l'espace},
        date={1934},
        ISSN={0001-5962},
     journal={Acta Math.},
      volume={63},
      number={1},
       pages={193\ndash 248},
         url={http://dx.doi.org/10.1007/BF02547354},
      review={\MR{1555394}},
}

\bib{R01}{book}{
      author={Robinson, James~C.},
       title={Infinite-dimensional dynamical systems},
      series={Cambridge Texts in Applied Mathematics},
   publisher={Cambridge University Press},
     address={Cambridge},
        date={2001},
        ISBN={0-521-63204-8},
         url={http://dx.doi.org/10.1007/978-94-010-0732-0},
        note={An introduction to dissipative parabolic PDEs and the theory of
  global attractors},
      review={\MR{1881888 (2003f:37001a)}},
}

\bib{T84}{book}{
      author={Temam, Roger},
       title={Navier-{S}tokes equations},
   publisher={AMS Chelsea Publishing, Providence, RI},
        date={2001},
        ISBN={0-8218-2737-5},
        note={Theory and numerical analysis, Reprint of the 1984 edition},
      review={\MR{1846644 (2002j:76001)}},
}

\end{biblist}
\end{bibdiv}

\end{document}